\documentclass[11pt]{amsart}
\usepackage[all]{xy}
\usepackage{amssymb}
\usepackage{amsthm}
\usepackage[normalem]{ulem}
\usepackage{amsmath,mathtools}
\usepackage{amscd,enumitem}
\usepackage{caption}
\usepackage[justification=centering]{caption}
\usepackage{verbatim}
\usepackage{eurosym}
\usepackage{float}
\usepackage{color}
\usepackage{dcolumn}
\usepackage[mathscr]{eucal}
\usepackage[all]{xy}
\usepackage{bbm}
\usepackage[textheight=8.7in, textwidth=6.7in]{geometry}
\newtheorem*{conj*}{Conjecture}
\newtheorem{theorem}{Theorem}[section]

\theoremstyle{definition}
\newtheorem*{remark}{Remark}

\theoremstyle{plain}

\newtheorem{lemma}[theorem]{Lemma}

\newtheorem{prop}[theorem]{Proposition}

\newcommand{\Z}{\mathbb{Z}}

\newcommand{\mo}{\mathrm{mo}}

\newcommand{\N}{\mathbb{N}}

\DeclareMathOperator\Log{Log}
\numberwithin{equation}{section}

\newtheoremstyle{example}
  {\topsep}   
  {\topsep}   
  {\normalfont}  
  {0pt}       
  {\bfseries} 
  {.}         
  {5pt plus 1pt minus 1pt} 
  {}          
\theoremstyle{example}
\newtheorem*{example}{Example}

\def\({\left(}
\def\){\right)}

\def\Li{\mathrm{Li}_2}

\usepackage{centernot}

\begin{document}
\title[Hook length distributions]{Distributions of Hook lengths in integer partitions}
\author{Michael Griffin, Ken Ono,  and Wei-Lun Tsai}

\dedicatory{In memory of Christine Bessenrodt}

\address{Department of Mathematics, Brigham Young University, Provo, UT 84602}
\email{mjgriffin@math.byu.edu}

\address{Department of Mathematics, University of Virginia, Charlottesville, VA 22904}
\email{ko5wk@virginia.edu}

\address{Department of Mathematics, University of Virginia, Charlottesville, VA 22904}
\email{wt8zj@virginia.edu}


\thanks{K.O. thanks  the Thomas Jefferson Fund and the NSF
(DMS-2002265 and DMS-2055118) for their support.}

\begin{abstract} Motivated by the many roles that hook lengths play in mathematics, we study
the distribution of the number of $t$-hooks in the partitions of $n$.
We prove that the limiting distribution is normal with
 mean $\mu_t(n)\sim \frac{\sqrt{6n}}{\pi}-\frac{t}{2}$ and variance $\sigma_t^2(n)\sim \frac{(\pi^2-6)\sqrt{6n}}{2\pi^3}.$ 
 Furthermore, we prove that the distribution of the number of hook lengths that are multiples of a fixed $t\geq 4$ in partitions of $n$ converge to a
 shifted Gamma distribution  with parameter $k=(t-1)/2$ and scale $\theta=\sqrt{2/(t-1)}.$
     \end{abstract}

\maketitle

\section{Introduction and statement of results}

The study of the statistical properties of partitions and their Young diagrams is rich with deep results. Works by  Pittel \cite{Pittel}, Szalay and Tur\'an \cite{ST1, ST2, ST3}, Temperley \cite{Temperley}, and Vershik \cite{Vershik} form a large body of work on questions related to the expected ``limiting shapes'' of Young diagrams (see the more recent paper by Bogachev \cite{Bogachev}  for more recent results in the field). In this paper we study the statistical properties of the hook lengths in Young diagrams of integer partitions. In this regard, there is recent work by Mutafchiev \cite{Mutafchiev} concerning the expected hook length of a randomly chosen cell. Here we study a different aspect.
To make this precise we first recall some notation.

A {\it partition} $\lambda=(\lambda_1, \lambda_2,\dots, \lambda_m)$ of $n$, denoted $\lambda \vdash n$,  is a nonincreasing sequence  of positive integers that sum to $n$. Its {\it Young diagram}  is the left-justified array of boxes where the row lengths are the parts.  
 The {\it hook} ${H}(k,j)$
of the cell in position $(k,j)$ is the set of cells below or to the right of that cell, including the cell itself, and the {\it hook length} $h(k,j):=(\lambda_k-k)+(\lambda'_j-j)+1,$ is the number of cells in the hook ${H}(k,j)$. 
Here $\lambda'_j$ is the number of boxes in column $j$, which is the same as the number of parts of the partition that are at least $j$.

\begin{center}
\includegraphics[height=18mm]{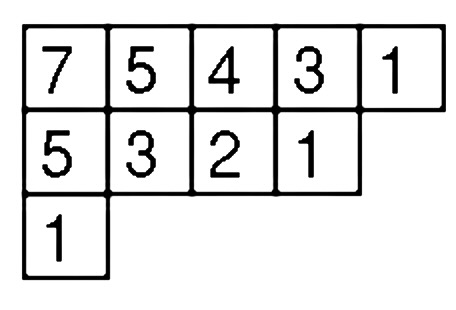}
\captionof{figure}{Hook lengths for $\lambda=(5,4,1)$}\label{figure0}
\end{center}
\smallskip

Multisets $\mathcal{H}(\lambda)$ of partition hook lengths  have many roles in combinatorics, number theory, and representation theory
(e.g. \cite{AndrewsEriksson, KerberJames, Sagan}). For instance, {\color{black} a standard Young tableaux for a partition $\lambda$ of $n$ is obtained by writing the numbers $1$ through $n$ in the boxes of the Young diagram so that each column and each row forms an increasing sequence.} The Frame-Robinson-Thrall hook length formula
$$
d_{\lambda}=\frac{n!}{\prod_{h\in \mathcal{H}(\lambda)} h}
$$
gives
the number of standard Young tableaux for  $\lambda.$ This
  is also the degree 
of  the canonical irreducible representation of the symmetric group $S_n$ associated to $\lambda.$  
As another important example, we have the famous Nekrasov-Okounkov  identity (see (6.12) of \cite{NekrasovOkounkov})\footnote{This formula was also obtained by Westbury (see Proposition 6.1 and 6.2 of \cite{Westbury}).} 
\begin{equation}\label{NO}
\sum_{\lambda} q^{|\lambda|} \prod_{h\in \mathcal{H}(\lambda)} \left(1-\frac{z}{h^2}\right)
=\prod_{n=1}^{\infty}\left(1-q^n\right)^{z-1},
\end{equation}
which arises in combinatorics, mathematical physics and the theory of modular forms.

In this paper, we study the numbers $Y_t(n)$ which count the $t$-hooks (i.e. hooks of length $t$) among all partitions of $n$.  For fixed $t$, we derive the limiting behavior of the sequence $\{Y_t(n)\}$ for $n\in \N$, and we give asymptotics for the accumulation function

\begin{equation}\label{DistFunc1}
D_t(k;n):= \frac{\# \left \{ \lambda \vdash n \ {\text {\rm with $\leq k$ many hook lengths of size $t$}}\right\}}{p(n)}.
\end{equation}

\begin{theorem}\label{MainTheorem1} If $t$ is a fixed positive integer, then the following are true for the sequence $\{Y_t(n)\}.$

\noindent
(1)  The sequence is asymptotically normal with mean
$\mu_t(n)\sim \frac{\sqrt{6n}}{\pi}-\frac{t}{2}$ and variance  $\sigma_t^2(n)\sim \frac{(\pi^2-6)\sqrt{6n}}{2\pi^3}.$

\smallskip
\noindent
(2) If we  let $k_{t,n}(x):=\mu_t(n) + \sigma_t(n)x,$ then in terms of Gauss's error function $E(x)$ we have
 $$
 \lim_{n\rightarrow +\infty}D_t(k_{t,n}(x);n) =\frac{1}{\sqrt{2\pi}}\int_{-\infty}^{x}e^{-\frac{y^2}{2}}dy=:E(x).
 $$
\end{theorem}

\begin{remark}
The $t=1$ case of Theorem~\ref{MainTheorem1} recovers a result by
Brennan, Knopfmacher and Wagner \cite{BKW} on the distribution of ascents in partitions, as this number
equals the number of size 1-hooks.
\end{remark}

\begin{example}
Theorem~\ref{MainTheorem1} asserts that the limiting distribution of 2-hooks is a normal distribution with
mean
$\mu_2(n)\sim \frac{\sqrt{6n}}{\pi}-1$
and variance  $\sigma_2^2(n)\sim \frac{(\pi^2-6)\sqrt{6n}}{2\pi^3}.$
For $n=5000$, we find that
$$
\sum_{\lambda \vdash 5000} T^{\# \left\{ 2\in \mathcal{H}(\lambda)\right\}}=
704T+9211712T^2+\dots+1805943379138T^{98}+2T^{99}.
$$
Figure~\ref{figure1}  plots $Y_2(5000)$.

 \medskip
 \begin{center}
\includegraphics[height=58mm]{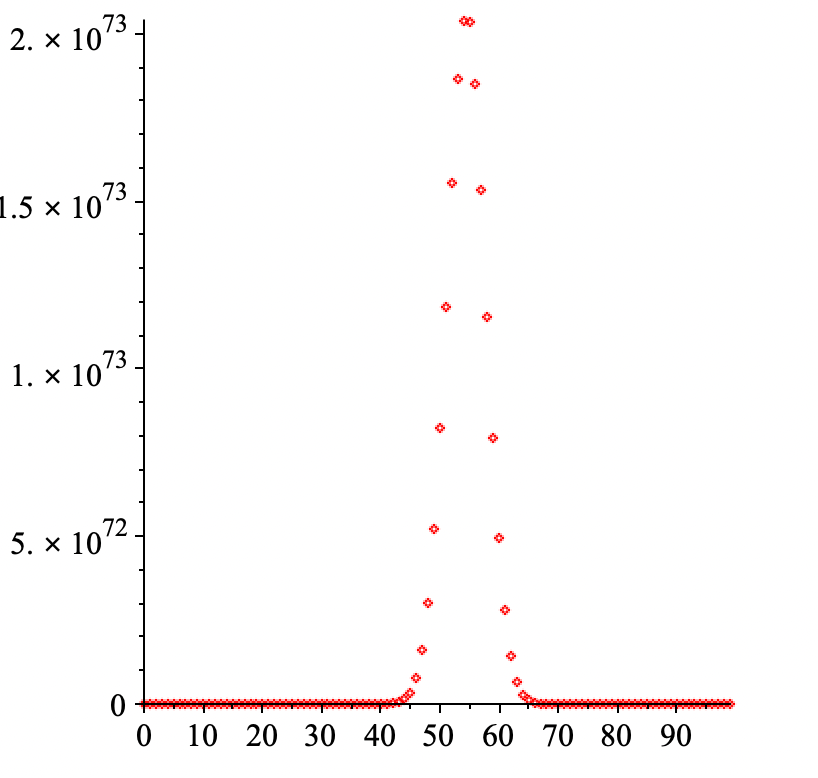}
\captionof{figure}{$Y_2(5000)$}\label{figure1}
\end{center}
\medskip

\noindent
Table~\ref{table1} illustrates the cumulative distribution approximation
$D_2(k_{2,5000}(x);5000) \approx E(x).$
\smallskip

\begin{center}
\begin{tabular}{|c|cc|cc|cc|}
\hline \rule[-3mm]{0mm}{8mm}
$x$       && $D_2(k_{2,5000}(x),5000)$           && $E(x)$  & $D_2(k_{2,5000}(x),5000)/E(x)$ & \\   \hline 
$-1.5$ && $0.0658\dots$  && $0.0668\dots$ & $0.9849\dots$&\\
$\vdots$ && $\vdots$  && $\vdots$ & $\vdots$&\\
$ 0.0$ &&  $0.5055\dots$ && $0.5000\dots$ & $1.0011\dots$ & \\
$ 1.0$ &&  $0.8246\dots$ && $0.8413\dots$ & $0.9802\dots$ & \\
$ 2.0$ &&  $0.9685\dots$ && $0.9772\dots$ & $0.9911\dots$ & \\
\hline
\end{tabular}
\captionof{table}{Asymptotics for the cumulative distribution for $n=5000$}\label{table1}
\end{center}
\end{example}

We next consider the sequence  $\{\widehat{Y}_t(n)\}$ of distributions of  the number of hook lengths in $t\N$  among the partitions of size $n.$
This question is motivated by work of Han that extends  (\ref{NO}) by giving infinite families of modular forms with level structure and cuspidal divisor.
If $\mathcal{H}_t(\lambda)$ is the multiset of hook lengths of $\lambda$ that are in $t\N$, then he proved (see Theorem 1.3 of \cite{Han}) that
$$
\sum_{\lambda} q^{|\lambda|} \prod_{h\in \mathcal{H}_t(\lambda)} \left(y-\frac{tyz}{h^2}\right)=
\prod_{n=1}^{\infty} \frac{(1-q^{tn})^t}{(1-(yq^t)^n)^{t-z}(1-q^n)}.
$$

For $t\geq 4,$ we prove that the limiting distribution is a  shifted Gamma distribution
with parameter $k=(t-1)/2$ and scale $\theta=\sqrt{2/(t-1)},$ and we determine asymptotics for the cumulative distribution
\begin{equation}\label{DistFcn2}
\widehat{D}_t(k;n):= \frac{\# \left \{ \lambda \vdash n \ {\text {\rm with $\leq k$ many hook lengths in  $t\N$}}\right\}}{p(n)}.
\end{equation}
 Recall (e.g. II.2 of \cite{Feller}) that a random variable $X_{k,\theta}$ satisfies the {\it Gamma distribution with parameter $k>0$ and scale $\theta>0$} if its probability distribution function is 
$F_{k,\theta}(x):=\frac{1}{\Gamma(k)\theta^k}\cdot x^{k-1}e^{-\frac{x}{\theta}}.$

\begin{theorem}\label{MainTheorem2}
If $t\geq 4,$ then the following are true for the sequence $\{\widehat{Y}_t(n)\}.$

\noindent
(1) The sequence satisfies

$$
 \widehat{Y}_t(n) \sim   \frac{n}{t}-\frac{\sqrt{3(t-1)n}}{\pi t} \cdot X_{\frac{t-1}{2},\sqrt{\frac{2}{t-1}}},
$$ 
and has mean
$\widehat{\mu}_t(n)\sim \frac{n}{t}-\frac{(t-1)\sqrt{6n}}{2\pi t},$ mode $\widehat{\mo}_t(n)\sim \frac{n}{t}-\frac{(t-3)\sqrt{6n}}{2\pi t},$ and variance  $\widehat{\sigma}_t^2(n)\sim \frac{3(t-1)n}{\pi^2t^2}.$

\smallskip
\noindent
(2) If we  let $\widehat{k}_{t,n}(x):=\widehat{\mu}_t(n) + \widehat{\sigma}_t(n)x,$ then in terms of the lower incomplete gamma function we have
 $$
 \lim_{n\rightarrow +\infty}\widehat{D}_t(\widehat{k}_{t,n}(x);n) =\frac{\gamma\left(\frac{t-1}{2};\sqrt{\frac{t-1}{2}}x+\frac{t-1}{2}\right)}{\Gamma\left(\frac{t-1}{2}\right)}.
 $$
\end{theorem} 

\begin{remark} The proof of Theorem~\ref{MainTheorem2}, which uses  properties of Gamma distributions with
 $k>1,$
does not apply for $t\in \{2,3\}$ as $(t-1)/2 \leq 1.$
Indeed, the  $\{\widehat{Y}_2(n)\}$ and $\{\widehat{Y}_3(n)\}$ do not even have  continuous limiting distributions.
 The fact that 100\% of $n$ do not have a 2-core or 3-core partition \cite{GranvilleOno} implies that these distributions
 are populated with many vanishing terms as illustrated by
$$
\sum_{\lambda \vdash 19} T^{\# \mathcal{H}_2(\lambda)}=300T^9 + 185T^8 + 5T^2.
$$
\end{remark}

\begin{example}
Theorem~\ref{MainTheorem2} gives
 $\widehat{Y}_{11}(n) \sim \frac{n}{11}-\frac{\sqrt{30n}}{11\pi}\cdot X_{5,\frac{\sqrt{5}}{5}},$
with mean $\widehat{\mu}_{11}(n)\sim \frac{n}{11}-\frac{5\sqrt{6n}}{11\pi}$ and variance
 $\widehat{\sigma}_{11}^2(n)\sim \frac{30n}{121\pi^2}.$
 Figure~\ref{figure2}  gives $\widehat{Y}_{11}(1000).$

 \smallskip
 \begin{center}
\includegraphics[height=67mm]{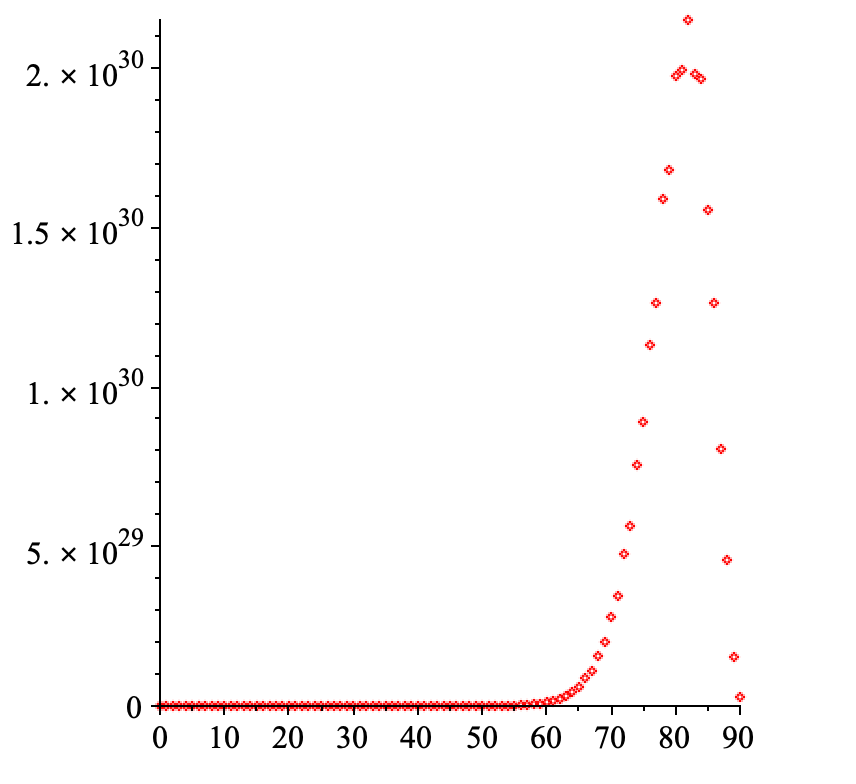}
\captionof{figure}{$\ \ \ \widehat{Y}_{11}(1000)$}\label{figure2}
\end{center}
\smallskip
 
\noindent
Table~\ref{table2} illustrates the approximation  $\widehat{D}_{11}({k}_{11,1000}(x);1000) \approx \frac{\gamma\left(5;\sqrt{5}x+5\right)}{24}=:\widehat{E}_{11}(x).$
\smallskip

\begin{center}
\begin{tabular}{|c|cc|cc|cc|}
\hline \rule[-3mm]{0mm}{8mm}
$x$       && $\widehat{D}_{11}({k}_{11,1000}(x);1000)$           && $\widehat{E}_{11}(x)$  & $\widehat{D}_{11}({k}_{11,1000}(x);1000)/\widehat{E}_{11}(x)$ & \\   \hline 
$-1.00$ && $0.1319\dots$  && $0.1467\dots$ & $0.8993\dots$&\\
$\vdots$ && $\vdots$  && $\vdots$ & $\vdots$&\\
$ 0.75$ &&  $0.7410\dots$ && $0.7954\dots$ & $0.9315\dots$ & \\
$ 1.00$ && $0.8226\dots$ && $0.8474\dots$ &$ 0.9707\dots$ &\\
$ 1.25$ &&  $0.8872\dots$ && $0.8880\dots$ & $0.9991\dots$ & \\
\hline
\end{tabular}
\captionof{table}{Asymptotics for the cumulative distribution for $n=1000$}\label{table2}
\end{center} 
\end{example}
\smallskip

This paper is organized as follows.
 In Section~\ref{NutsAndBolts} we recall work of Han that offers the relevant enumerative generating functions, and we then determine
their asymptotics  via the saddle point method, with assistance from the  Euler-Maclaurin summation formula. In Section~\ref{Proofs} we use these asymptotics to compute the moments of these statistics, which in turn imply Theorems~\ref{MainTheorem1} and \ref{MainTheorem2} thanks to a classical theorem of Curtiss.

\section*{Acknowledgements} \noindent
The authors thank George Andrews, Kathrin Bringmann, Richard Stanley and Ole Warnaar for their valuable correspondence on this project. Finally, they thank the referees for their careful reading of the original submission and for their helpful suggestions.

\section{Nuts and Bolts}\label{NutsAndBolts}

We recall work of Han on the enumeration of hook lengths, and we derive important propositions
(see Proposition~\ref{Asymptotics1} and \ref{Asymptotics2})  that are central to the proof of Theorems~\ref{MainTheorem1} and \ref{MainTheorem2}.
Han obtained (see Thm. 1.4 and Cor. 5.1 of \cite{Han}) the following important generating functions for each fixed positive integer $t:$
\begin{equation}\label{thookGenFcn}
G_t(T;q)=
\sum_{n=0}^{\infty} P_t(n;T)q^n=\sum_{m,n} p_t(m;n)T^mq^n
:=\sum_{\lambda}q^{|\lambda|} T^{\# \left\{ t\in \mathcal{H}(\lambda)\right\}}
=\prod_{n=1}^{\infty}\frac{(1+(T-1)q^{tn})^t}{1-q^n},
\end{equation}
\begin{equation}\label{tmulthookGenFcn}
\widehat{G}_t(T;q)=\sum_{n=0}^{\infty} \widehat{P}_t(n;T)q^n=\sum_{m,n} \widehat{p}_t(m;n)T^mq^n:=\sum_{\lambda} q^{|\lambda|}T^{\# \mathcal{H}_t(\lambda)}=
\prod_{n=1}^{\infty} \frac{(1-q^{tn})^t}{(1-(Tq^t)^n)^t(1-q^n)}.
\end{equation}

\noindent
The next  two propositions
on $P_t(n;T)$ and $\widehat{P}_t(n;T)$ are the main results of this section.

\begin{prop}\label{Asymptotics1}
Suppose that $\eta \in (0,1]$ and
  $\eta\leq T\leq\eta^{-1}.$   If $c(T):=\sqrt{\pi^2/6-\Li(1-T)},$ then 
$$
{P}_t(n;T)=\frac{c(T)}{2\sqrt{2} \pi n T^{\frac{t}{2}}}\cdot e^{c(T)\left( 2\sqrt{n}-\frac{1}{\sqrt{n}}\right)}\cdot\left(1+O_{\eta}(n^{-\frac{1}{7}})\right),
$$
where $\Li(z):=-\int_{0}^{z}\frac{\log(1-u)}{u} du$ is the dilogarithm function.
\end{prop}

\noindent
The next proposition is more subtle, and pertains  to suitable real sequences.

\begin{prop}\label{Asymptotics2}
If $t$ is a positive integer and $T:=\{T_n\}$ is a positive real sequence  for which
$T_n = e^{\frac{\alpha(T)+\varepsilon_T(n)}{\sqrt{n}}},$
where $\alpha(T)$ is real and $\varepsilon_T(n)=o_T(1),$
then
\begin{displaymath}
\begin{split}
&\widehat{P}_t(n;T_n)=\\
&\ \frac{1}{2^{\frac{7}{4}}3^{\frac{1}{4}}n}\cdot\sqrt{\frac{1}{\sqrt{6}}+\frac{\alpha(T)+\varepsilon_T(n)}{\pi t}}\left(\frac{\pi t}{\pi t+\sqrt{6}\left(\alpha(T)+\varepsilon_T(n)\right)}\right)^{\frac{t}{2}}\cdot e^{\pi\sqrt{n}\left(\sqrt{\frac{2}{3}}+\frac{\alpha(T)+\varepsilon_T(n)}{\pi t}\right)}\cdot(1+O_{T}(n^{-\frac{1}{7}})).
\end{split}
\end{displaymath} 
\end{prop}

\subsection{Proof of Proposition~\ref{Asymptotics1}}

The proof of Proposition~\ref{Asymptotics1} requires the next lemma.

\begin{lemma}\label{DiLogEstimates}
If $\eta\in(0,1],$ then for  $0<\alpha<1$ and $\eta\leq T\leq\eta^{-1}$ we have
\begin{align}
&\sum_{j=1}^{\infty}\log(1-e^{-j\alpha})=-\frac{\pi^2}{6\alpha}-\frac{1}{2}\log\left(\frac{\alpha}{2\pi}\right)+O(\alpha)\label{eq:a0},\\
    &\sum_{n=1}^{\infty}\frac{t^2n(T-1)}{T-1+e^{tn\alpha}}=-\frac{\Li(1-T)}{\alpha^2}+O_{\eta}(1)\label{eq:a},\\
     &\sum_{n=1}^{\infty}\log\left(1+(T-1)e^{-tn\alpha}\right)=-\frac{\Li(1-T)}{t\alpha}-\frac{1}{2}\log T+O_{\eta}(\alpha),\label{eq:b}\\
      &\sum_{n=1}^{\infty}\frac{t^3n^2e^{-tn\alpha}}{(1+(T-1)e^{-tn\alpha})^2}=-\frac{2}{\alpha^3}\frac{\Li(1-T)}{T-1}+O_{\eta}(\alpha).\label{eq:c}
    \end{align}
\end{lemma}

\begin{proof} 
For $f\in C^{j+1}([a,b])$ and $a,b\in\Z$, Euler-Maclaurin summation 
(e.g.  Thm. 2.1.9 of \cite{Murty}) gives
$$
    \sum_{a<n\leq b}f(n)=\int_{a}^{b}f(x)dx+\sum_{r=0}^{j}\frac{(-1)^{r+1}}{(r+1)!}\left(f^{(r)}(b)-f^{(r)}(a)\right)B_{r+1}+\frac{(-1)^j}{(j+1)!}\int_{a}^bB_{j+1}(x-\lfloor x\rfloor)f^{(j+1)}(x)dx,
$$
where $B_r(x)$ is the $r$th Bernoulli polynomial and $B_r:=B_r(0).$  Letting $a=0$ and $j=0$  gives
\begin{align*}
    \sum_{n=1}^{b}\frac{t^2n(T-1)}{T-1+e^{tn\alpha}}&=\int_{0}^{b}\frac{t^2(T-1)x}{T-1+e^{t\alpha x}}dx+\frac{t^2(T-1)b}{2(T-1+e^{t\alpha b})}\\
    &\hspace{1.in}+\int_{0}^bB_1(x-\lfloor x\rfloor)\frac{t^2(T-1)(T-1+e^{t\alpha x})-t^3\alpha(T-1)xe^{t\alpha x}}{(T-1+e^{t\alpha x})^2}dx\\
    &=\frac{1}{\alpha^2}\left[\Li((1-T)e^{-t\alpha b})-\Li(1-T)-t\alpha\log\left((T-1)e^{-t\alpha b}+1\right)\right]+O_{\eta}\left(\frac{b^3}{e^{t\alpha b}}\right).
\end{align*}
To obtain \eqref{eq:a}, we let $b\rightarrow\infty$ and find that
$$
\sum_{n=1}^{\infty}\frac{t^2n(T-1)}{T-1+e^{tn\alpha}}=-\frac{\Li(1-T)}{\alpha^2}+O_{\eta}(1).
$$
Applying Euler-Maclaurin summation proves the other cases  {\it mutatis mutandis}.
\end{proof}

\begin{proof}[Proof of Proposition~\ref{Asymptotics1}]
We first note that (\ref{thookGenFcn}) implies that 
\begin{equation}\label{integralrep}
P_t(n;T)=\frac{1}{2\pi}\int_{-\pi}^{\pi}(ze^{ix})^{-n}G_t(T;ze^{ix})dx=\frac{1}{2\pi}\int_{-\pi}^{\pi}e^{g_t(T;ze^{ix})}dx,
\end{equation}
where $g_t(T;w):=\Log(w^{-n}G_t(T;w))$ for $0<|w|<1.$
To apply the  saddle point method, we must determine $z=e^{-\alpha}$ for $\alpha>0,$ such that $g_t'(T;z)=0$ (throughout we consider the derivative with respect to the second parameter). By (\ref{thookGenFcn}),  this  is equivalent to
$$
    \sum_{j=1}^{\infty}\frac{t^2j(T-1)}{T-1+e^{tj\alpha}}+\sum_{j=1}^{\infty}\frac{j}{e^{j\alpha}-1}=n.
$$
By combining \eqref{eq:a} with
\begin{equation}\label{eqtwo}
\sum_{j=1}^{\infty}\frac{j}{e^{j\alpha}-1}=\frac{\pi^2}{6\alpha^2}-\frac{1}{2\alpha}+O(1),
\end{equation}
which holds for $0<\alpha<1,$ we find that
\begin{equation}\label{alpha}
\alpha=c(T)\cdot n^{-\frac{1}{2}}-\frac{1}{4}n^{-1}+O_\eta(n^{-\frac{3}{2}}).
\end{equation}

We now estimate $g_t(T;z), g_t''(T;z),$ and $g_t'''(T;z).$  Plugging $z=e^{-\alpha}$ into $g_t(T;z)$, we obtain
$$
    g_t(T;z)=t\sum_{j=1}^{\infty}\log\left(1+(T-1)e^{-tj\alpha}\right)-\sum_{j=1}^{\infty}\log(1-e^{-j\alpha})+n\alpha.
$$
 Therefore, (\ref{eq:a0}), (\ref{eq:b}) and (\ref{alpha}) gives
\begin{equation}\label{eval2}
    g_t(T;z)=2c(T)\sqrt{n}+\frac{1}{2}\log\left(\frac{c(T)}{2\pi T\sqrt{n}}\right)+O_{\eta}\left (n^{-\frac{1}{2}}\right).
\end{equation}
Similarly, by using \eqref{eq:a} and \eqref{eq:c} we obtain
\begin{equation}\label{eqthree}
\sum_{j=1}^{\infty}\frac{j^2e^{-j\alpha}}{(1-e^{-j\alpha})^2}=\frac{\pi^2}{3\alpha^3}-\frac{1}{2\alpha^2}+O(\alpha),
\end{equation}
which implies that
\begin{align}\label{2derivative}
  g_t''(T;z)&=\left[n+\sum_{j=1}^{\infty}\frac{t^3j^2e^{-tj\alpha}}{(1+(T-1)e^{-tj\alpha})^2}
  -\sum_{j=1}^{\infty}\frac{t^2j(T-1)}{T-1+e^{tj\alpha}}
  +\sum_{j=1}^{\infty}\frac{j^2e^{-j\alpha}}{(1-e^{-j\alpha})^2}
  -\sum_{j=1}^{\infty}\frac{j}{e^{j\alpha}-1}\right]e^{2\alpha}\notag\\
  &=e^{2c(T)n^{-\frac{1}{2}}+O_{\eta}(n^{-1})}\left(\frac{2}{c(T)}n^{\frac{3}{2}}+O_{\eta}(n)\right).
\end{align}
By the same argument, thanks to (\ref{alpha}), we find that 
\begin{equation}\label{3derivative}
g_t'''(T;z)=O_{\eta}\left(\sum_{j=1}^{\infty}\frac{j^3e^{-j\alpha}}{(1-e^{j\alpha})^4}\right)=O_{\eta}(\alpha^{-4})=O_{\eta}\left (n^2\right).
\end{equation}

To complete the proof, we now let $P_t(n;T)=I+II,$ where
\begin{align*}
    I:=\frac{1}{2\pi}\int_{|x|\leq n^{-{5}/{7}}}e^{g_t(T;ze^{ix})}dx \ \ \ \ {\text {\rm and}}\ \ \ \ 
    II:=\frac{1}{2\pi}\int_{|x|> n^{-{5}/{7}}}e^{g_t(T;ze^{ix})}dx.
\end{align*}
To estimate $I,$ we use the Taylor expansion of $g_t(T;w)$ centered at the saddle point $z=e^{-\alpha}$
\begin{equation*}
    g_t(T;w)=g_t(T;z)+\frac{g_t''(T;z)(w-z)^2}{2}+O_{\eta}(g_t'''(T;z))(w-z)^3).
\end{equation*}
Since $|x|\leq n^{-5/7},$ estimate (\ref{alpha}) gives
$$w-z=ze^{ix}-z=e^{-\alpha}\left (ix+O\left (x^2\right )\right)=\left(1+O_{\eta}\left (n^{-\frac{1}{2}}\right)\right)\left(ix+O\left(n^{-\frac{10}{7}}\right)\right)=ix+O_{\eta}\left(n^{-\frac{17}{14}}\right).$$
Therefore, we obtain
\begin{equation}\label{Taylorest}
     g_t(T;w)=g_t(T;z)-\frac{g_t''(T;z)(x)^2}{2}+O_{\eta}\left(n^{-\frac{1}{7}}\right).
\end{equation}
Combining \eqref{eval2}, \eqref{2derivative}, \eqref{3derivative}, and \eqref{Taylorest}, we obtain the main term asymptotic
\begin{eqnarray}
I&=\frac{e^{g_t(T;z)}}{2\pi}\left[\int_{-\infty}^{\infty}e^{-\frac{g_t''(T;z)x^2}{2}}dx-\int_{|x|> n^{-{5}/{7}}}e^{-\frac{g_t''(T;z)x^2}{2}}dx\right]\cdot\left (1+O_{\eta}\left(n^{-\frac{1}{7}}\right)\right)\notag\\
&=\frac{c(T)}{2\sqrt{2}\pi n T^{\frac{t}{2}}}\cdot e^{c(T)\left( 2\sqrt{n}-\frac{1}{\sqrt{n}}\right)}\cdot\left(1+O_{\eta}\left(n^{-\frac{1}{7}}\right)\right).\label{Iestimate}
\end{eqnarray}

To estimate the tail error term $II,$ we  estimate $\frac{G_t(T;ze^{ix})}{G_t(T;z)}$ using
$$
e^{g_t(T;ze^{ix})}=e^{g_t(T;z)}\frac{G_t(T;ze^{ix})}{G_t(T;z)}.
$$
Since $T>0$, letting $w=ze^{ix}$ gives
\begin{align}\label{ratioest}
    \left|\frac{G_t(T;w)}{G_t(T;z)}\right|^2
       &\leq\prod_{j=1}^{\infty}\mathrm{Max}\left\{1,\left|\frac{1+(T-1)w^{j}}{1+(T-1)z^{j}}\right|^{2}\right\}\left|\frac{1-z^j}{1-w^j}\right|^2\notag\\
    &\leq \prod_{j=1}^{\infty}\mathrm{Max}\left\{1,\left(1+\frac{2z^j(1-T)(1-\cos(xj))}{(1-z^{j})^2}\right)\right\}\left(1+\frac{2z^j(1-\cos(xj))}{(1-z^j)^2}\right)^{-1}\notag\\
    &\leq \prod_{\sqrt{n}\leq j\leq 2\sqrt{n}}\mathrm{Max}\left\{1,\left(1+\frac{2z^j(1-T)(1-\cos(xj))}{(1-z^{j})^2}\right)\right\}\left(1+\frac{2z^j(1-\cos(xj))}{(1-z^j)^2}\right)^{-1}.
\end{align}
To reduce to the finite product in the last line,
 we used the fact that  for all $j\geq 1$ we have
$$\mathrm{Max}\left\{1,\left(1+\frac{2z^j(1-T)(1-\cos(xj))}{(1-z^{j})^2}\right)\right\}\left(1+\frac{2z^j(1-\cos(xj))}{(1-z^j)^2}\right)^{-1}\leq1.
$$ 

We consider two cases (i.e. $T>1$ and $T\leq1$) to estimate  \eqref{ratioest}. 
If $T>1$ and $j\in[\sqrt{n},2\sqrt{n}],$ then by  (\ref{alpha}) we have $2z^j/(1-z^j)^2\leq c_{\eta},$ for some $c_{\eta}>0.$ This implies that
\begin{equation}\label{ratioest2}
     \left|\frac{G_t(T;w)}{G_t(T;z)}\right|^2\leq \prod_{\sqrt{n}\leq j\leq 2\sqrt{n}}\left(1+c_{\eta}(1-\cos(xj))\right)^{-1}.
\end{equation}
 If $T\leq 1$, then we have 
$$\mathrm{Max}\left\{1,\left(1+\frac{2z^j(1-T)(1-\cos(xj))}{(1-z^{j})^2}\right)\right\}=1+\frac{2z^j(1-T)(1-\cos(xj))}{(1-z^{j})^2}.
$$ 
Moreover, we have $2z^j/(1-z^{j})^2\leq c_{\eta}.$ A similar calculation also shows that \eqref{ratioest2} still holds for $T\leq1$ by choosing a suitable $c_{\eta}>0.$

We divide the range of $x$ into two cases $n^{-5/7}\leq|x|\leq \frac{\pi}{2\sqrt{n}},$ and $\frac{\pi}{2\sqrt{n}}\leq|x|\leq{\pi}.$ For the first case, we can use the inequality $1-\cos(xj)\geq\frac{2}{\pi^2}(xj)^2$
to estimate \eqref{ratioest2}, giving
\begin{equation}\label{ratioest3}
    \left|\frac{G_t(T;w)}{G_t(T;z)}\right|^2\leq \prod_{\sqrt{n}\leq j\leq 2\sqrt{n}} \left(1+\frac{2c_{\eta}}{\pi^2}(xj)^2\right)^{-1}\ll e^{-c_{\eta}(x^2n^{\frac{3}{2}})}\ll e^{-c_\eta n^{\frac{1}{14}}}.
\end{equation}
In the case where $\frac{\pi}{2\sqrt{n}}\leq|x|\leq{\pi},$ we count the $j\in [\sqrt{n},2\sqrt{n}]$ for which there is an $\ell\in\Z$ with
$-n^{-\frac{1}{12}}+2\ell\pi\leq xj\leq n^{-\frac{1}{12}}+2\ell\pi.$
The total number of such $j$ is $\gg n^{1/2}+O(n^{5/12}).$ Hence, we have
\begin{equation}\label{ratioest4}
    \left|\frac{G_t(T;w)}{G_t(T;z)}\right|^2\leq  \left(1+c_{\eta}(1-\cos(n^{-\frac{1}{12}}))\right)^{-(n^{\frac{1}{2}}+O(n^{\frac{5}{12}}))}\ll e^{-c_\eta n^{\frac{1}{14}}}.
\end{equation}
 By combining \eqref{ratioest3} and \eqref{ratioest4}, we obtain the upper bound for the tail 
\begin{align*}
    II\ll\frac{1}{2\pi}\int_{|x|> n^{-{5}/{7}}}e^{g_t(T;z)}\left|\frac{G_t(T;ze^{ix})}{G_t(T;z)}\right|dx\ll_{_\eta}e^{-\frac{c(T)}{\sqrt{n}}-\frac{c_{\eta}}{2}\cdot n^{\frac{1}{14}}}.
\end{align*}
As $P_t(n;T)=I+II,$ the proposition follows from this inequality and (\ref{Iestimate}).

\end{proof}

\subsection{Proof of Proposition~\ref{Asymptotics2}}
For each positive integer $n$,  (\ref{tmulthookGenFcn}) implies that 
\begin{equation}\label{integralrep2}
\widehat{P}_t(n;T_n)=\frac{1}{2\pi}\int_{-\pi}^{\pi}(ze^{ix})^{-n}\widehat{G}_t(T_n;ze^{ix})dx=\frac{1}{2\pi}\int_{-\pi}^{\pi}e^{\widehat{g}_t(T_n;ze^{ix})}dx,
\end{equation}
where $\widehat{g}_t(T_n;w):=\Log(w^{-n}\widehat{G}_t(T_n;w))$ for $0<|w|<1.$ 
We aim to locate the saddle point $z=e^{-\beta_n},$ with $\beta_n>0.$ To this end, we solve
$$
  -\sum_{j=1}^{\infty}\frac{t^2j}{e^{tj\beta_n}-1}+\sum_{j=1}^{\infty}\frac{t^2jT_n^j}{e^{tj\beta_n}-T_n^j}+\sum_{j=1}^{\infty}\frac{j}{e^{j\beta_n}-1}=n.
$$
By \eqref{eqtwo} and the definition of $\alpha(T)$ and $\varepsilon_T(n)$, we obtain 
we find that
\begin{equation}\label{beta}
\beta_n=\left (\frac{\pi}{\sqrt{6}}+ \frac{\alpha(T)+\varepsilon_T(n)}{t}\right) \cdot n^{-\frac{1}{2}}+O_T\left(n^{-1}\right).
\end{equation}
Since we have $\varepsilon_T(n)=o_T(1),$ it follows that
\begin{equation}\label{beta}
\beta_n =\left (\frac{\pi}{\sqrt{6}}+\frac{\alpha(T)}{t} \right) \cdot n^{-\frac{1}{2}} +o_T\left(n^{-\frac{1}{2}}\right).
\end{equation}
We now estimate $\widehat{g}_t(T_n;z),~\widehat{g}_t''(T_n;z),$ and $\widehat{g}_t'''(T_n;z).$  Plugging $z=e^{-\beta_n}$ into $\widehat{g}_t(T_n;z)$, we obtain
$$
    \widehat{g}_t(T_n;z)=t\sum_{j=1}^{\infty}\log\left (1-e^{-tj\beta_n}\right)-t\sum_{j=1}^{\infty}\log\left(1-T_n^je^{-tj\beta_n}\right)-\sum_{j=1}^{\infty}\log\left(1-e^{-j\beta_n}\right)+n\beta_n.
$$
 Applying (\ref{eq:a0}) to all three terms gives
\begin{equation}\label{eval2b}
    \widehat{g}_t(T_n;z)=\frac{t\pi^2}{6(t\beta_n-\log T_n)}+\frac{1}{2}\log\left(\frac{\beta_n}{2\pi }\right)+\frac{t}{2}\log\left(\frac{t\beta_n-\log T_n}{t\beta_n }\right)+n\beta_n+O_{T}(\beta_n).
\end{equation}
Similarly, by using \eqref{eqtwo} and \eqref{eqthree} we obtain
\begin{align}\label{2derivative2b}
  \widehat{g}_t''(T_n;z)&=\left[n+\frac{\pi^2t^3}{3(\beta_n t-\log T_n)^3}+\frac{t-1}{2\beta_n^2}-\frac{(t^3+\pi^2t^2)}{2(\beta_n t-\log T_n)^2}+\frac{1-t}{2\beta_n}+\frac{t^2}{2(\beta_n t-\log T_n)} + O_{T}(\beta_n)\right]e^{2\beta_n}.
\end{align}
By the same argument, thanks to (\ref{beta}),  we find that 
\begin{equation}\label{3derivative2b}
\widehat{g}_t'''(T_n;z)=O_{T}\left((\beta_n t-\log T_n)^{-4}\right)=O_{T}\left(\beta_n^{-4}\right).
\end{equation}

Arguing as in the proof of Proposition~\ref{Asymptotics1} with \eqref{beta}, \eqref{eval2b}, \eqref{2derivative2b}, and \eqref{3derivative2b}, we obtain
\begin{align*}
&\widehat{P}_t(n;T_n)=\frac{e^{\widehat{g}_t(T_n;z)}}{2\pi}\cdot\int_{-\infty}^{\infty}e^{-\frac{\widehat{g}_t''(T_n;z)x^2}{2}}dx\cdot(1+O_{T}(n^{-\frac{1}{7}}))
=\frac{e^{\widehat{g}_t(T_n;z)}}{2\pi}\cdot\sqrt{\frac{2\pi}{|\widehat{g}_t''(T_n;z)|}}\cdot\left(1+O_{T}\left(n^{-\frac{1}{7}}\right)\right)\\
&\ \ = \frac{1}{2^{\frac{7}{4}}3^{\frac{1}{4}}n}\cdot\sqrt{\frac{1}{\sqrt{6}}+\frac{\gamma(T)+\alpha_T(n)}{\pi t}}\left(\frac{\pi t}{\pi t+\sqrt{6}\left(\alpha(T)+\varepsilon_T(n)\right)}\right)^{\frac{t}{2}}\cdot e^{\pi\sqrt{n}\left(\sqrt{\frac{2}{3}}+\frac{\alpha(T)+\varepsilon_T(n)}{\pi t}\right)}\cdot\left (1+O_{T}\left (n^{-\frac{1}{7}}\right) \right).
\label{Iestimate}
\end{align*}
This completes the proof of the proposition.

\section{Proofs of Theorems~\ref{MainTheorem1} and \ref{MainTheorem2}}\label{Proofs}

We prove Theorems~\ref{MainTheorem1} and \ref{MainTheorem2} using the
method of moments, where the crucial device is the following classical theorem of Curtiss.

\begin{theorem}[Theorem 2 of \cite{Curtiss}]
Let $\left\{X_n\right\}$ be a sequence of  real random variables. Then define the corresponding moment generating function
$$
M_{X_n}(r):=\int_{-\infty}^{\infty}e^{rx}dF_{n}(x),
$$
where $F_{n}(x)$ is the cumulative distribution function associated with $X_n.$ If the sequence $\left\{M_{X_n}(r)\right\}$ converges pointwise on a neighborhood of $r = 0,$ then $\left\{X_n\right\}$ converges in distribution.
\end{theorem}

\begin{proof}[Proof of Theorem~\ref{MainTheorem1}]
For each $n\geq 1,$  we consider the $r$th power moment 
\begin{equation}\label{mgf1}
M(Y_t(n); r):=\frac{1}{p(n)}\sum_{m=0}^{\infty}p_t(m;n)\cdot e^{\frac{(m-\mu_t(n))r}{\sigma_t(n)}}.
\end{equation}
By Curtiss's Theorem, combined with the  theory of normal distributions,  it suffices to show that
\begin{equation}\label{Thm1Criterion}
\lim_{n\rightarrow +\infty}M(Y_t(n); r)=e^{\frac{r^2}{2}}.
\end{equation}

By evaluating  $P_t(n;T)$ at $T=1$ (i.e. $P_t(n;1)=p(n)$) and $e^{\frac{r}{\sigma_t(n)}},$ we have
$$
M(Y_t(n); r)=\frac{P_t(n;e^{\frac{r}{\sigma_t(n)}})}{p(n)}\cdot e^{-\frac{\mu_t(n)}{\sigma_t(n)}r}.
$$
Proposition \ref{Asymptotics1} gives
\begin{equation}\label{mgfest}
M(Y_t(n); r)=\frac{c(e^{\frac{r}{\sigma_t(n)}})\cdot\left(1+O_{\eta}\left(n^{-\frac{1}{7}}\right)\right)}{c(1)\cdot\left(1+O_{}\left(n^{-\frac{1}{7}}\right)\right)}\cdot e^{-\frac{t}{2\sigma_t(n)}r-\frac{\mu_t(n)}{\sigma_t(n)}r+(2n^{\frac{1}{2}}-n^{-\frac{1}{2}})\cdot(c(e^{\frac{r}{\sigma_t(n)}})-c(1))}.
\end{equation}
Since $e^{\frac{r}{\sigma_t(n)}}> 0$ and $e^{\frac{r}{\sigma_t(n)}}\rightarrow1,$ as $n\rightarrow\infty,$ the implied constant can be chosen to be independent of $\eta.$ 
By direct calculation of the dilogarithm function, we find that
$c(1)=\pi/\sqrt{6},$ and 
$$
c(e^{\frac{r}{\sigma_t(n)}})=\frac{\pi}{\sqrt{6}}+\sqrt{\frac{3}{2}}\frac{1}{\pi}\left(\frac{r}{\sigma_t(n)}\right)+\sqrt{\frac{3}{2}}\frac{(\pi^2-6)}{4\pi^3}\left(\frac{r^2}{\sigma_t^2(n)}\right)+O\left(\frac{r^3}{\sigma_t^3(n)}\right).
$$
Therefore, \eqref{mgfest} becomes
\begin{align*}
M(Y_t(n); r)
&=\left(1+O_r(n^{-\frac{1}{7}})\right)\cdot e^{\left(-\frac{t}{2}-\mu_t(n)+\frac{\sqrt{6n}}{\pi}\right)\frac{r}{\sigma_t(n)}+\sqrt{6n}\left(\frac{\pi^2-6}{4\pi^3}\right)\left(\frac{r^2}{\sigma_t^2(n)}\right)+O_r\left(n^{-\frac{3}{4}}\right)}\\
&=\left(1+O_r(n^{-\frac{1}{7}})\right)\cdot e^{\frac{r^2}{2}+o_r(1)}.
\end{align*}
Letting $n\rightarrow +\infty,$ we obtain (\ref{Thm1Criterion}) confirming Theorem~\ref{MainTheorem1}.
\end{proof}
\begin{proof}[Proof of Theorem~\ref{MainTheorem2}]
To prove Theorem~\ref{MainTheorem2}, we recall that if $k>1$ and $r<1/\theta,$ then the moment generating function for the random variable
$X_{k,\theta}$ is (for example, see II.2 of \cite{Feller})
$$
M(X_{k,\theta}; r)=\frac{1}{(1-\theta r)^k}.
$$
This distribution has mean $\mu_{k,\theta}=k\theta$, mode $\mo_{k,\theta}=(k-1)\theta$, and variance $\sigma^2_{k,\theta}=k\theta^2.$
If $a$ and $b$ are real, then the shifted Gamma distribution
$a X_{k,\theta} +b$ has moment generating function
$$
M(aX_{k,\theta}+b;r)= e^{br}\cdot M(X_{k,\theta},ar)=\frac{e^{br}}{(1-\theta ar)^k},
$$
and has mean $ak\theta+b$, mode $a(k-1)\theta+b$, and
variance $a^2k\theta^2.$
We compare $\widehat{Y}_t(n)$ with
$a X_{k,\theta}+b,$ where $(k,\theta):=\left(\frac{t-1}{2},\sqrt{\frac{2}{t-1}}\right),$
 and 
$a:=-1$ and $b:=\sqrt{2(t-1)}/2.$ Therefore, we assume that $(t-1)/2 >1$, which is equivalent to $t\geq 4.$

To apply Curtiss's theorem, we compute the moment generating function as in (\ref{mgf1}), with the
 claimed mean $\widehat{\mu}_t(n)\sim n/t-(t-1)\sqrt{6n}/2\pi t,$ and variance $\widehat{\sigma}_t(n)\sim \sqrt{3(t-1)n}/\pi t.$ 
Applying Proposition~\ref{Asymptotics2} with $\alpha(T):={\pi t r}/{\sqrt{3(t-1)}}$ and $T_n:=e^{\frac{\alpha(T)}{\sqrt{n}}}$, we find  that
\begin{align*}
M(\widehat{Y}_t(n); r)&= \frac{(2^{\frac{7}{4}}3^{\frac{1}{4}}n)^{-1}\cdot\sqrt{\frac{1}{\sqrt{6}}+\frac{r}{\sqrt{3(t-1)}}}\cdot\left(1+\sqrt{\frac{2}{t-1}}r\right)^{-\frac{t}{2}}\cdot\left(1+O_r(n^{-\frac{1}{7}})\right)}{(4\sqrt{3}n)^{-1}\cdot\left(1+O(n^{-\frac{1}{7}})\right)}\cdot e^{\frac{n}{t\widehat{\sigma}_t(n)}r-\frac{\widehat{\mu}_t(n)}{\widehat{\sigma}_t(n)}r}.\\
&=\frac{e^{\frac{\sqrt{2(t-1)}}{2}r}}{\left(1+\sqrt{\frac{2}{t-1}}r\right)^{\frac{t-1}{2}}}\cdot\left(1+O_r(n^{-\frac{1}{7}})\right).
\end{align*} 
Therefore, Curtiss's theorem gives $\widehat{Y}_t(n)\sim a\widehat \sigma_t(n)X_{k,\theta}+b\widehat \sigma_t(n)+\widehat \mu_t(n),$ as well as the claimed mean, mode and variance.
To obtain claim (2), we recall that if
 $k>1,$ then the Gamma distribution $X_{k,\theta}$  has cumulative distribution function (e.g. II.2 of \cite{Feller})
$D_{k,\theta}(x)=\gamma\left(k;\frac{x}{\theta}\right)/\Gamma(k),$
where $\gamma(\alpha;x)$ is the lower incomplete Gamma function. 
\end{proof}

\end{document}